\documentclass{amsart}
\usepackage[centertags]{amsmath}
\usepackage{amsfonts}
\usepackage{amssymb}
\usepackage{amsthm}
\usepackage[colorlinks]{hyperref}
\usepackage{tikz}
\usepackage{tikz-cd}
\usepackage[normalem]{ulem}
\usepackage[shortlabels]{enumitem} 
\usepackage{marginnote}
\usepackage[all]{xy}
\usepackage{cancel}
\usepackage{bm}
\usepackage{tikz}
\usepackage{tikz-cd}
\usetikzlibrary{arrows}
\usetikzlibrary{arrows.meta}
\usetikzlibrary{cd}
\usepackage{array,multirow,makecell}

\newcommand{\fonction}[5]{\begin{array}{l rcl}
#1: & #2 & \longrightarrow & #3\\
      & #4 & \longmapsto & #5 \end{array}}

\newcommand{\N}{\mathbb{N}}

\newcommand{\ep}{\varepsilon}
\newcommand\restr[2]{{
		\left.\kern-\nulldelimiterspace 
		#1 
		\littletaller 
		\right|_{#2} 
}}
\newcommand{\littletaller}{\mathchoice{\vphantom{\big|}}{}{}{}}

\newcommand{\Lip}{{\mathrm{Lip}}}

\newcommand{\F}{\mathcal{F}}

\def\<{\langle}
\def\>{\rangle}

\renewcommand{\hat}{\widehat}

\def\restriction#1#2{\mathchoice
              {\setbox1\hbox{${\displaystyle #1}_{\scriptstyle #2}$}
              \restrictionaux{#1}{#2}}
              {\setbox1\hbox{${\textstyle #1}_{\scriptstyle #2}$}
              \restrictionaux{#1}{#2}}
              {\setbox1\hbox{${\scriptstyle #1}_{\scriptscriptstyle #2}$}
              \restrictionaux{#1}{#2}}
              {\setbox1\hbox{${\scriptscriptstyle #1}_{\scriptscriptstyle #2}$}
              \restrictionaux{#1}{#2}}}
\def\restrictionaux#1#2{{#1\,\smash{\vrule height .8\ht1 depth .85\dp1}}_{\,#2}}

\theoremstyle{plain}
\newtheorem{thm}{Theorem}[section]

\newtheorem{cor}[thm]{Corollary}

\newtheorem{lem}[thm]{Lemma}
\newtheorem{lemma}[thm]{Lemma}
\newtheorem{prop}[thm]{Proposition}

\theoremstyle{definition}
\newtheorem{defn}[thm]{Definition}

\newtheorem{remark}[thm]{Remark}

\newtheorem{question}{Question}

\author[M. Lemay]{Mathis Lemay}
\address[M. Lemay]{Univ Gustave Eiffel, Univ Paris Est Creteil, CNRS, LAMA UMR8050, F-77447 Marne-la-Vallée, France}
\email{mathis.lemay@univ-eiffel.fr}

\begin{document}
	\title[Equivalences for weithted Lipschitz operators]{Equivalences between certain properties of weighted Lipschitz operators}
	


	\begin{abstract}
       We show that for a weighted Lipschitz operator $\omega\widehat{f}$, certain linear properties are equivalent. Specifically, we prove that compactness, strict singularity, and strict cosingularity are all equivalent to the property of not fixing any complemented copy of $\ell^1$. Then we generalize this result to operators between Lipschitz-free spaces that preserve finitely supported elements, a larger class of operators.
	\end{abstract}

	\maketitle

\section{Introduction}\label{intro}

Let $(M,d)$ be a pointed metric space, with basepoint $0_M \in M$. Consider the vector space
$$\mathrm{Lip}_0(M)=\{f: M \longrightarrow \mathbb{C} ~  | ~  f(0_M)=0, \;  f \text{ is Lipschitz}\},$$ which is a Banach space endowed with the classical norm on Lipschitz function spaces:
$$\|f\|_L=\displaystyle\sup_{x \neq y} \frac{|f(x)-f(y)|}{d(x,y)}.$$  
The mapping

$$\begin{array}{lcccccc}
\delta_M : & M & \longrightarrow & \Lip_0(M)^* & & &  \\
           & x & \longmapsto     & \delta_M(x) : &\Lip_0(M) & \longrightarrow & \mathbb C\\
           &   &                 &               & f       &\longmapsto& f(x)
\end{array}
$$
is an isometry. Moreover, $\{\delta_M(x) ~  | ~ x \in M\}$ is linearly independent, and the Lipschitz-free space is defined as $\mathcal{F}(M)=\overline{\mathrm{span}}\{\delta_M(x), x \in M\}$, where the closure is taken with respect to the norm topology of $\mathrm{Lip}_0(M)^*$. The fundamental universal extension property of Lipschitz-free spaces (see \cite{Godefroy_Kalton}) states that, given a Lipschitz map $f: M \longrightarrow X$ from a pointed metric space $M$ to a Banach space $X$, such that $f(0_M)=0_X$, there exists a unique bounded linear map $\overline{f}: \mathcal{F}(M) \longrightarrow X$ such that $\overline{f}(\delta_M(x))=f(x)$. Moreover, $\|\overline{f}\|=\|f\|_L$, the best Lipschitz constant of $f$. In particular, this result implies that $\mathcal{F}(M)$ serves as the natural predual of $\mathrm{Lip}_0(M)$. Furthermore, given a Lipschitz map between two pointed metric spaces $f: M \longrightarrow N$ which fixes the origin, there exists a unique bounded linear operator $\widehat{f}: \mathcal{F}(M) \longrightarrow \mathcal{F}(N)$ such that $\delta_N \circ f=\widehat{f} \circ \delta_M$ and $\|\widehat{f}\| = \|f\|_L$.
A similar operator can be constructed by introducing a weight $\omega: M \longrightarrow \mathbb{R}$ as follows:
$$\omega\widehat{f} : \delta(x) \in \F(M) \longmapsto w(x) \delta(f(x)) \in \F(N).$$

However, the fundamental theorem mentioned above does not guarantee the continuity of the linear operator. Nevertheless, in \cite{ACP24}, the authors have established precise conditions ensuring that $\omega\widehat{f}$ is bounded, as well as criteria for its compactness. The present work establishes a connection between compactness, strict singularity and strict cosingularity. The latter notion has been introduced by Pełczyński in 1965 (see \cite{Pel}). Notably, Pe\l{}czy\'nski proved in \cite{Pel} the following result:

\begin{thm}{(Pe\l{}czy\'nski, \cite{Pel})}
Let $X$ a Banach space and $(\Omega,\mathcal{T},\mu)$ a measured space. Let $T: X \longrightarrow L^1(\mu)$ an operator. The following properties are equivalent: 
\begin{enumerate}
    \item[-] $T$ is strictly cosingular;
    \item[-] $T$ is weakly compact;
    \item[-] $T$ is strictly singular;
    \item[-] $T$ does not fix any complemented copy of $\ell^1$.
\end{enumerate}
\end{thm}
The connection between the preceding theorem and our framework of \emph{weighted Lipschitz operators} may not be immediately obvious, so we briefly clarify this link. It is well known that Lipschitz-free spaces exhibit a strong structural relationship with $L^1(\mu)$ spaces. In particular, it was shown in \cite{Cuth_Johanis} that $\mathcal{F}(M)$ contains a complemented copy of $\ell^1$ whenever $M$ is infinite. Moreover, by a result of Godard \cite{Godard}, $L^1([0,1])$ embeds isometrically into $\mathcal{F}(M)$ whenever $M$ contains a bi-Lipschitz copy of a subset $S \subset \mathbb{R}$ with positive Lebesgue measure. 
Furthermore, the paper \cite{p1u} characterizes those Lipschitz-free spaces that possess the Schur property, a feature reflecting a pronounced $\ell^1$-type behavior within the Banach space. 
\smallskip

Therefore, inspired by Pe\l{}czy\'nski's result above, we will prove the following main result of this note. 

\begin{thm}\label{thm:main}
Let $M,N$ be infinite metric spaces, $f: M \longrightarrow N$ and $\omega: M \longrightarrow \mathbb{C}$ such that 
$\omega \widehat{f} : \delta(x) \in \F(M) \mapsto \omega(x)\delta(f(x)) \in \F(N)$
is a bounded operator.\\ The following properties are equivalent: 
\begin{enumerate}
    \item[(i)] $\omega\widehat{f}$ is compact;
    \item[(ii)] $\omega\widehat{f}$ is weakly compact;
    \item[(iii)] $\omega\widehat{f}$ is super strictly singular;
    \item[(iv)] $\omega\widehat{f}$ is super strictly cosingular;
    \item[(v)] $\omega\widehat{f}$ is strictly singular;
    \item[(vi)] $\omega\widehat{f}$ is strictly cosingular;
    \item[(vii)] $\omega\widehat{f}$ does not fix any copy of $\ell^1$;
    \item[(viii)] $\omega\widehat{f}$ does not fix any complemented copy of $\ell^1$.
\end{enumerate}
\end{thm}
It is well known that the adjoint of $\omega \widehat{f}$ is the weigthed composition operator $\omega C_f : g \in \Lip_0(N) \longmapsto \omega g \circ f \in \Lip_0(M)$.
Thus, we deduce a corollary which states similar equivalences for the adjoint operator: the properties (i) to (vi) are still equivalent for $\omega C_f$. They are also equivalent to (vii)' and (viii)' which consist in replacing $\ell^1$ by $\ell^{\infty}$ in (vii) and (viii). 
\smallskip

In the third part of this article, we show that this theorem holds for a larger class of operators. More specifically, we consider the operators $T : \mathcal{F}(M) \longrightarrow \mathcal{F}(N)$ which preserve finitely supported elements in a certain way, such as Lipschitz operators or multiplication operators.

\section{Preliminaries}

\subsection{Notation} 

Throughout the paper, $M$ and $N$ denote metric spaces. 
\smallskip

For $x \in M$ and $r>0$, $B(x,r):=\{y \in M ~  | ~  d(y,x)<r\}$ denotes the open ball centered at $x$ of radius $r$, and $\overline{B}(x,r)$ denotes the closed ball centered at $x$ of radius~$r$. 
\smallskip

If $A \subset M$ and $x \in M$, then the distance from $x$ to $A$ is defined by: 
$$d(x,A):=\inf\{d(x,y) ~  | ~  y \in A\}.$$
Moreover, $\mathrm{rad}(M):=\sup\limits_{y \in M} d(y,0)$ denotes the radius of $M$.
\smallskip

If $(x_n) \subset M$ is a sequence in $M$ and $\varepsilon>0$, we say that $(x_n)$ is \emph{$\varepsilon$-separated} if: 
$$\forall n \neq m \in \mathbb{N}, d(x_n,x_m) \geq \varepsilon.$$

Throughout this paper, the letters $X$ and $Y$ are reserved for Banach spaces, and $S_X$ stands for the unit sphere of $X$. The notation $X \cong Y$ means that $X$ is isomorphic to $Y$.
 Unless otherwise stated, the term \emph{operator} refers to a bounded linear map.
\smallskip

We say that an operator $T: X \longrightarrow Y$ \emph{does not fix any (complemented) copy of $\ell^1$} if it is impossible to find subspaces $E \subset X$ and $F \subset Y$ such that $E \cong \ell^1, F \cong \ell^1$, ($E$ is complemented in $X$, $F$ is complemented in $Y$), and $\restriction{T}{E}: E \longrightarrow F$ is an isomorphism. Note that saying that $T$ fixes a complemented copy of $\ell^1$ is equivalent to say that there exist subspaces $E \subset X$ and $F \subset Y$, continuous projections $P_X : X \longrightarrow E$ and $P_Y : Y \longrightarrow F$ such that the following diagram:

$$\begin{tikzcd}
X \arrow[d,"P_X"'] \arrow[r,"T"] & Y \arrow[d,"P_Y"]\\
E \arrow[r,"\restriction{T}{E}"'] & F
\end{tikzcd}$$
commutes, and $\restriction{T}{E}$ is an isomorphism. Indeed, assuming that $T$ fixes a complemented copy of $\ell^1$, there exists a continuous projection $P_Y : Y \longrightarrow F$, and then define $P_X$ by $P_X=T \circ P_Y \circ (\restriction{T}{E})^{-1}$. One can easily check that $P_X : X \longrightarrow E$ is a projection and that it makes the diagram commutative.

\subsection{Lipschitz-free spaces} For $M,N$ metric spaces, $\mathcal{F}(M)$, $\mathcal{F}(N)$ denote their associated Lipschitz-free spaces. The notation $\mathrm{Lip}_0(M)$ is used with the same meaning as in the introduction.

Denote by
$$\mathcal{M}:=\left\{\frac{\delta(x)-\delta(y)}{d(x,y)} ~  | ~  x \neq y \in M\right\} \subset \F(M)$$ the set of \emph{molecules}.
\smallskip

For $k \in \mathbb{N}$, $\mathcal{FS}_k(M)$ denotes the set of the elements of $\mathcal{F}(M)$ whose support contains at most $k$ elements (see \cite{ACP24} for the definition of the support of an element $\gamma \in \mathcal{F}(M)$ and related results). It follows from \cite{ACP21} that $\mathcal{FS}_k(M)$ is weakly closed and hence closed in the norm topology. 
According to Theorem C in \cite{ACP21}, $\mathcal{FS}_k(M)$ satisfies a version of the Schur property (although it is not a Banach space): for any sequence $(\gamma_n)_n \subset \mathcal{FS}_k(M)$,
\[
\gamma_n \overset{w}{\longrightarrow} \gamma \iff \gamma_n \overset{\|\cdot\|}{\longrightarrow} \gamma,
\]
where $w$ denotes the weak topology.
\smallskip

For an arbitrary map $f: M \longrightarrow N$, $C_f$ denotes the composition operator induced by $f$. The composition operator $C_f$ is defined by:
$$\begin{array}{lcccccc}
C_f : & \mathbb{C}^N & \longrightarrow & \mathbb{C}^M & & &  \\
           & g & \longmapsto     & g \circ f.
\end{array}
$$
The composition operators appear in a very natural way as the adjoint of Lipschitz operators since: 
$$\langle (\widehat{f})^*(g),\delta(x)\rangle=\langle g,\widehat{f}(\delta(x))\rangle=\langle g,\delta(f(x))\rangle=g(f(x)), ~  g \in \mathrm{Lip}_0(N)$$
Hence, for $f \in \mathrm{Lip}_0(M)$, $(\widehat{f})^*=C_f$ seen as an operator from $\Lip_0(N)$ to $\Lip_0(M)$.
\smallskip

If $N \subset M$ with $0 \in N$, then, in the real case (i.e the maps $f$ have real values), $\mathcal{F}(N)$ is a subspace of $\mathcal{F}(M)$. In the complex case, $\mathcal{F}(N)$ is isomorphic to a subspace of $\mathcal{F}(M)$.

Recall, for all practical purposes, the McShane-Whitney's extension theorem:

\begin{lemma}[McShane-Whitney's extension theorem]\label{lemma:McShane}
    Let $E$ be a subset of $M$ et $g : E \longrightarrow \mathbb{R}$ be a Lipschitz map. There exists $f : M \longrightarrow \mathbb{R}$ Lipschitz such that $\restriction{f}{E}=g$ and $\|f\|_L=\|g\|_L$.
\end{lemma}

If the function $g$ is complex valued, extending real and imaginary parts yields an extension $f$ with Lipchitz norm bounded above by $\sqrt{2}\|g\|_L$.

Finally, let us recall that if $M$ is not complete, where $\overline{M}$ is its completion, then $\mathcal{F}(M)$ and $\mathcal{F}(\overline{M})$ are isometrically isomorphic.

\subsection{Weighted Lipschitz operators} 
This article deals with weighted Lipschitz operators, so let us recall how they are built. First, let us fix a map for $\omega: M \longrightarrow \mathbb{C}$ referred to as the weight. 
\begin{defn} For $f : M \longrightarrow N$ any map, we define the weighted Lipschitz operator associated to $f$:
\[
\fonction{\omega\widehat{f}}{\mathrm{span}(\delta(M))}{\mathrm{span}(\delta(N))}{\gamma=\sum\limits_{i=1}^{n} a_i\delta(x_i)}{\sum\limits_{i=1}^{n} a_i\omega(x_i)\delta(f(x_i)).}
\]
\end{defn}
\begin{defn} We also define the weighted composition operators induced by $f: M \longrightarrow N$ a map between two pointed metric spaces:

$$\fonction{\omega C_f }{\mathrm{Lip}_0(N)}{\mathrm{Lip}_0(M)}{g}{x \in M \longmapsto \omega(x)g(f(x)).}$$
\end{defn}

In \cite{ACP24}, the following result is shown: 

\begin{thm}{(Abbar, Coine, Petitjean, \cite{ACP24})}
Let $M,N$ be pointed metric spaces. Let $\omega: M \longrightarrow \mathbb{C}$ and $f: M \longrightarrow N$ be any maps such that $f(0_M)=0_N$ or $\omega(0_M)=0$. The following assertions are equivalent: 
\begin{enumerate}
    \item[(i)] $\omega\widehat{f}$ extends to a bounded operator from $\mathcal{F}(M)$ to $\mathcal{F}(N)$;
    \item[(ii)] 
    $\omega C_f$ defines a bounded operator from $\mathrm{Lip}_0(N)$ to $\mathrm{Lip}_0(M)$ and $(\omega\widehat{f})^*=\omega C_f$;
    \item[(iii)] For all $g \in \mathrm{Lip}_0(N)$, $\omega C_f(g) \in \mathrm{Lip}_0(M)$;
    \item[(iv)] $\varphi: x \in M \longmapsto \omega(x)\delta(f(x)) \in \mathcal{F}(N)$ is Lipschitz. 
\end{enumerate}
\end{thm}

One of the main objectives of this article is to study operator properties which are closely related to norm or weak compactness. To that end, the following characterization, established in \cite{ACP24}, will prove handy.

\begin{prop}{(Abbar, Coine, Petitjean, \cite{ACP24})}\label{prop:equivWeak}
Let $M,N$ be metric spaces with distinguished points $0_M \in M$ and $0_N \in N$. Let $f: M \longrightarrow N$ and $\omega: M \longrightarrow \mathbb{R}$ satisfy $f(0_M)=0_N$ or $\omega(0_M)=0$. If $\omega\widehat{f}$ is a bounded operator, then it is (weakly) compact if and only if 
\[
\omega\widehat{f}(\mathcal{M})=\left\{\frac{\omega(x)\delta(f(x))-\omega(y)\delta(f(y))}{d(x,y)} ~  | ~  x \neq y \in M\right\}
\]
is relatively (weakly) compact.
\end{prop}

A metric characterization of compactness in the non-weighted case is also proved in \cite{ACP21}. Moreover, some sufficient conditions are provided in \cite{ACP24} in the weighted case.

Finally, let us recall an important consequence of Proposition~\ref{prop:equivWeak} and the fact that $\mathcal{FS}_2(N)$ has the Schur property (see \cite{ACP24}): if $M,N$ are pointed complete metric spaces, and if $\omega\widehat{f}:\F(M) \to \F(N)$ is a weighted Lipschitz operator, then one has:
\[
\omega\widehat{f} \text{ is compact } \iff \omega\widehat{f} \text{ is weakly compact.}
\]

\subsection{Strictly singular and strictly cosingular operators} Finally, this article relates compactness to the notions of strictly singular and strictly cosingular operators. Let us recall the corresponding definitions below: 

\begin{defn}{}
Let $X,Y$ be Banach spaces. An operator $T: X \longrightarrow Y$ is strictly singular if there exists no infinite-dimensional subspace $E \subset X$ such that $\restriction{T}{E}$
is an isomorphism onto its range.
\end{defn}

\begin{defn}{}
Let $X,Y$ be Banach spaces. An operator $T: X \longrightarrow Y$ is strictly cosingular if, for every subspace $Z \subset Y$ such that $\dim\left(Y/Z\right)=\infty$, the map $Q \circ T: X \longrightarrow Y/Z$ is not surjective, where $Q: Y \longrightarrow Y/Z$ is the canonical quotient map.
\end{defn}

\begin{remark}\label{SCScomp}
    See Pełczyński's article \cite{Pel} for more details on strictly cosingular operators. Note that a compact operator is clearly strictly singular, and also strictly cosingular (it is a consequence of the open mapping theorem), but both converses are false in general.
\end{remark}

The notion of strictly cosingularity was introduced by Pełczyński, probably as an attempt to obtain a dual notion of strict singularity. For instance, the following result holds: 

\begin{prop}{}\label{propSS}
Let $X,Y$ be Banach spaces and $T: X \longrightarrow Y$. Then, 
\begin{enumerate}
    \item[a)] if the adjoint operator $T^*: Y^* \longrightarrow X^*$ is strictly singular (resp. cosingular), then $T$ is strictly cosingular (resp. singular);
    \item[b)] if $T$ is strictly cosingular and weakly compact, then $T^*$ is strictly singular;
    \item[c)] if $T$ is strictly singular and $X$ is reflexive, then $T^*$ is strictly cosingular.
\end{enumerate}
\end{prop}

In general, the converse implication of a) is not true, Pełczyński gave counter--examples in \cite{Pel} (see Examples 1 and 2 in Section~II.2).
In the articles \cite{Pli} and \cite{Flores}, the authors introduced stronger notions, leading this time to dual properties, with success.

\begin{defn}\label{SS}
Let $X$ and $Y$ be Banach spaces, and let $T: X \to Y$. We say that $T$ is \emph{super strictly singular if:}
\[
\forall \varepsilon>0, \exists N_{\varepsilon} \in \mathbb{N}^*, \forall E \subset X, \dim(E) \geq N_{\varepsilon}, \exists x \in S_E \text{ such that } \|T(x)\| \leq \varepsilon,
\]
that is,
\[
b_n(T):=\sup_{E \subset X, \dim(E)=n} \inf \{\|T(x)\|: x \in S_E\} \underset{n \rightarrow +\infty}{\longrightarrow} 0.
\]
\end{defn}

\begin{defn}\label{SCS}
Let $X$ and $Y$ be Banach spaces, and let $T: X \to Y$. We say that $T$ is \emph{super strictly cosingular} if:
\[
a_n(T):=\sup\{\alpha_E(T): E \subset Y, \dim(Y/E)=n\} \underset{n \rightarrow +\infty}{\longrightarrow} 0,
\]
where: 
\[
\alpha_E(T):=\sup\{\alpha \geq 0: \alpha B_{Y/E} \subset Q \circ T(B_X)\},
\]
and where $Q: Y \longrightarrow Y/E$ is the canonical quotient map.
\end{defn}

It is a routine exercise to verify that if $T$ is super strictly singular (respectively, super strictly cosingular), then $T$ is strictly singular (respectively, strictly cosingular). The converse implications, however, do not hold in general; (see \cite{Pli}). Moreover, it is immediate that every compact operator is both super strictly singular and super strictly cosingular, although the converse implications are again false in general.

\begin{remark}\label{SSCScomp}
    The following results are proved in \cite{Flores}.
\begin{enumerate}[leftmargin=*]
    \item[-] With the notations in Definition~\ref{SS} and Definition~\ref{SCS}, $a_n(T)=b_n(T^*)$ and $b_n(T)=a_n(T^*)$. In particular, 
    T is super strictly singular (resp. cosingular) if and only if $T^*$ is super strictly cosingular (resp. singular).

    \item[-] The prefix ``super'' is well chosen, as it aligns naturally with the terminology arising from the theory of ultraproducts of Banach spaces. Indeed, an operator $T : X \to Y$ is super strictly singular (respectively, super strictly cosingular) if and only if, for every ultrafilter $\mathcal U$, the induced ultra-operator $T_{\mathcal U} : X_{\mathcal U} \longrightarrow Y_{\mathcal U}$
    is strictly singular (respectively, strictly cosingular).

\end{enumerate}
\end{remark}

\section{Equivalent properties for \texorpdfstring{$\omega\widehat{f}$}{weighted Lipschitz operators}}
Several implications from Pe\l{}czy\'nski’s classical theorem are known to extend to the setting of weighted Lipschitz operators. As already mentioned, $\omega\widehat{f}$ is weakly compact if and only if it is compact. It is a standard fact that compactness implies both strict singularity and strict cosingularity, and that either of these properties entails not fixing any complemented copy of $\ell_1$. These considerations motivate our investigation of whether the equivalences obtained in Pe\l{}czy\'nski’s result continue to hold for the class of operators considered here. As shown by our main result, this is indeed the case. For the reader's convenience, we recall its statement below.

\begingroup
\renewcommand{\thethm}{\ref{thm:main}}
\begin{thm}
Let $M,N$ be infinite metric spaces, $f: M \longrightarrow N$ and $\omega: M \longrightarrow \mathbb{C}$ such that $\fonction{\omega\widehat{f}}{\mathcal{F}(M)}{\mathcal{F}(N)}{\delta(x)}{\omega(x)\delta(f(x))}$ is a bounded operator.\\ The following properties are equivalent: 
\begin{enumerate}
    \item[(i)] $\omega\widehat{f}$ is compact;
    \item[(ii)] $\omega\widehat{f}$ is weakly compact;
    \item[(iii)] $\omega\widehat{f}$ is super strictly singular;
    \item[(iv)] $\omega\widehat{f}$ is super strictly cosingular;
    \item[(v)] $\omega\widehat{f}$ is strictly singular;
    \item[(vi)] $\omega\widehat{f}$ is strictly cosingular;
    \item[(vii)] $\omega\widehat{f}$ does not fix any copy of $\ell^1$;
    \item[(viii)] $\omega\widehat{f}$ does not fix any complemented copy of $\ell^1$. 
\end{enumerate}
\end{thm}
\endgroup

The equivalence $(i) \iff (v) \iff (vii)$ was first established in an unpublished work by the authors of \cite{p1u}. We are grateful to them for allowing us to include their result in this paper. Incidentally, the argument we present below also yields an independent proof of these equivalences.

Thanks to Remarks \ref{SCScomp} and \ref{SSCScomp}, the following implications hold: 
$$(i) \Rightarrow (iii) \Rightarrow (v) \text{ and } (i) \Rightarrow (iv) \Rightarrow (vi)$$
Thanks to Theorem C in \cite{ACP24}, $(i) \iff (ii)$. Also it is clear that $(vi) \Rightarrow (viii)$ and $(vii) \Rightarrow (viii)$. To summarize, the following scheme of implications is already known thanks to previous results:

\begin{figure}[h]
  \centering
\begin{tikzcd}[column sep=1em, row sep=2em]
  & (i) \arrow[rr,Leftrightarrow,"\text{\cite{ACP24}}"] \arrow[dl,Rightarrow] \arrow[dr,Rightarrow] & & (ii) \\
  (iii) \arrow[d,Rightarrow] & & (iv) \arrow[d,Rightarrow] & \\
  (v) \arrow[d,Rightarrow] & & (vi) \arrow[d,Rightarrow] & \\
  (vii) \arrow[rr,Rightarrow] & & (viii) &
\end{tikzcd}
  \caption{Diagram of implications}
  \label{diagram}
\end{figure}
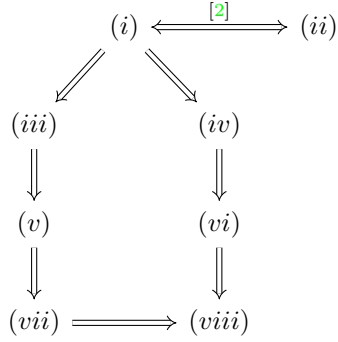

Therefore, it only remains to prove $(viii) \Rightarrow (i)$.
\medskip

By contraposition, let us assume that $\omega\widehat{f}$ is not compact. Thanks to Proposition~\ref{prop:equivWeak}, this means that $\omega\widehat{f}(\mathcal{M})$ is not relatively compact, where 
$$\mathcal{M}=\left\{m_{x,y}=\frac{\delta(x)-\delta(y)}{d(x,y)}~  | ~  x \neq y\right\}.$$ Hence, $\exists x_n,y_n \in M$ such that $(\gamma_n):=(\omega\widehat{f}(m_{x_n,y_n}))$ does not have any norm-convergent subsequence.
In particular, since $\mathcal{FS}_2(N)$ has the Schur property, $(\gamma_n)$ does not admit any weakly Cauchy subsequence. By Rosenthal's $\ell^1$ theorem (\cite{AK}, p.252), passing to a subsequence if necessary, we can assume that $(\gamma_n)$ is equivalent to the canonical basis of $\ell^1$. In order to show the desired implication, we will need the following proposition:

\begin{prop}{}\label{prop:complemented}
\label{prop:cl1}
Let $\gamma_n=a_n\delta(x_n)+b_n\delta(y_n) \in \mathcal{FS}_2(M)$ for all $n \in \mathbb{N}$. We assume that $(\gamma_n)$ is equivalent to the canonical basis of $\ell^1$. Then there exists a subsequence $(\gamma_{n_k})_k$ of $(\gamma_n)$ which is complemented in $\mathcal{F}(M)$.
\end{prop}

For the proof, we require the following lemmas:

\begin{lem}\label{lemma:Projection}
Let $(f_n)$ a bounded sequence in $\text{Lip}_0(M)$, with disjoint supports, that is: 
$$\exists C_1>0 \|f_n\|_L \leq C_1 , n \in \mathbb{N} \quad \text{ and } \quad \mathrm{supp}(f_n) \cap \mathrm{supp}(f_m)=\emptyset ~  \text{ for all } n \neq m.$$
Let $(\gamma_n)$ a bounded sequence of elements of $\mathcal{F}(M)$ which is  biorthogonal to $(f_n)$, i.e: 
$$\exists C_2>0, \|\gamma_n\| \leq C_2, n \in \mathbb{N} \quad \text{ and } \quad \langle f_n,\gamma_m\rangle=\delta_{n,m}=\begin{cases}
    1 \text{ if } n=m\\
    0 \text{ if } n \neq m
\end{cases}.$$
Then there exists a continuous projection from $\mathcal{F}(M)$ onto $\overline{\mathrm{span}}\{\gamma_n, n \in \mathbb{N}\}$.
\end{lem}

\begin{proof}
    Let us consider the map 
    $$\varphi: x \in M \longmapsto \displaystyle\sum_{n=1}^{+\infty} f_n(x)\gamma_n \in \overline{\mathrm{span}}\{\gamma_n, n \in \mathbb{N}\}.$$
    Note that, since the $f_n$'s have disjoint supports, this sum is well defined. Indeed, if $x \in M$, then either $x$ is in the support of one $f_{n_0}$ (and only one), and then the sum is just $\varphi(x)=f_{n_0}(x)\gamma_{n_0}$, or $x$ is not in the support of any of the $f_n$, and then $\varphi(x)=0$.
    Let $x,y \in M$. Then:
$$\|\varphi(x)-\varphi(y)\| \leq \sum_{n=1}^{+\infty} |f_n(x)-f_n(y)|\|\gamma_n\|.$$
Note that, since the $f_n$ have disjoint supports, this sum has at most two nonzero terms.
So: 
$$\|\varphi(x)-\varphi(y)\| \leq 2C_1C_2d(x,y).$$
Hence $\varphi \in \text{Lip}_0(M)$. We consider $P:=\overline{\varphi}$ obtained by the universal extension property of Lipschitz-free spaces (see Section~\ref{intro}). That is, $P: \F(M) \to \overline{\mathrm{span}}\{\gamma_n, n \in \mathbb{N}\}$ is such that: 
$$
P ( \delta(x) ) = \sum\limits_{n=1}^{+\infty} \langle f_n,\delta(x)\rangle\gamma_n.$$
The universal extension property ensures that $P$ is a bounded operator. Moreover, for $n \in \mathbb{N}$, 
$$P(\gamma_n)=\sum_{m=1}^{+\infty} \langle f_m,\gamma_n\rangle\gamma_m.$$
Since $\langle f_m,\gamma_n\rangle=\delta_{n,m}$, it follows that $P(\gamma_n)=\gamma_n$.

Finally, 
$$P \circ P(\delta(x))=\sum_{n=1}^{+\infty} \left\langle f_n,\sum_{m=1}^{+\infty} \langle f_m,\delta(x)\rangle\gamma_m\right\rangle\gamma_n=\sum_{n=1}^{+\infty} \sum_{m=1}^{+\infty} \langle f_m,\delta(x)\rangle\langle f_n,\gamma_m\rangle\gamma_n.$$
Again, since $\langle f_n,\gamma_m\rangle=\delta_{n,m}$, we have: 
$$P \circ P(\delta(x))=\sum_{n=1}^{+\infty} \langle f_n,\delta(x))\rangle\gamma_n=P(\delta(x)).$$
\end{proof}{}

\begin{lem}\label{lemma:sev}
Let $X$ a Banach space and $(x_n) \subset X$ a bounded 2-separated sequence. Let $F$ be a finite dimensional subspace of $X$. Then:
$$\forall \delta>0, \exists N \in \mathbb{N}, \forall n \geq N, d(x_n,F) >1-\delta.$$
\end{lem}

\begin{proof}
Let $\delta>0$. Since $(x_n)$ is bounded, one has: 
$$\exists C>0, \forall n \in \mathbb{N}, x_n \in B(0,C).$$
Let us denote $B:=B(0,C)$. Then it is easy to see that $d(x_n,F)=d(x_n,F \cap 3B)$ for every $n$.
Let us assume that for all $N \in \mathbb{N}$, there exists $n \geq N$ such that 
$$d(x_n,F) \leq 1-\delta.$$
Then, we can construct by induction a subsequence $(x_{n_k})$ such that for all $k \in \mathbb{N}, d(x_{n_k},F \cap 3B) \leq 1-\delta$. In addition, since $F \cap 3B$ is compact, 
$$\forall k \in \mathbb{N}, \exists y_k \in F \cap 3B, d(x_{n_k},F \cap 3B)=\|x_{n_k}-y_k\|.$$
By compactness, up to passing to a subsequence, we may assume that $y_k \longrightarrow y$ for some $y \in F \cap 3B$. Thus, for $k,m$ large enough, $\|y_k-y_m\| \leq \delta$. Hence
$$\|x_{n_k}-x_{n_m}\| \leq \|x_{n_k}-y_k\|+\|y_k-y_m\|+\|y_m-x_{n_m}\| \leq 1-\delta+\delta+1-\delta=2-\delta,$$
which contradicts the fact that $(x_n)$ is 2-separated.   
\end{proof}

We now have all the tools needed to complete the proof of the proposition, and hence of the theorem.

\begin{proof}[Proof of Proposition~\ref{prop:cl1}]
First and foremost, note that we may assume that $M$ is bounded. Indeed, if $M$ is an unbounded metric space, then, by \cite{AA}, there exists a bounded metric space $B(M)$ and an isomorphism 
\[
T: \mathcal{F}(M) \longrightarrow \mathcal{F}(B(M)),
\]
with the additional important property that $T(\mathcal{FS}_2(M)) = \mathcal{FS}_2(B(M))$. Consequently, if a projection $P_1$ onto $\mathcal{F}(B(M))$ exists as in the lemma, then the operator 
\[
P_2 = T^{-1} \circ P_1 \circ T
\]
provides the desired projection on $\mathcal{F}(M)$.

So let $M$ be a bounded metric space.
Let $(\gamma_n) \subset \mathcal{F}(M)$ be a sequence equivalent to the canonical basis of $\ell^1$, 
where each element is of the form
\[ \gamma_n = a_n \delta(x_n) + b_n \delta(y_n) \in \mathcal{FS}_2(M), \text{ for all } n \in \mathbb{N}.
\]
By hypothesis, $(\gamma_n)$ is equivalent to the canonical basis of $\ell^1$ so, up to multiplying by a positive constant, we may assume that: 
\[
(\ast) \quad \exists C>0, \ \forall a_1,\ldots,a_n \in \mathbb{C}, \ 
\sum_{i=1}^{n} |a_i| \leq \left\|\sum_{i=1}^{n} a_i\gamma_i\right\| \leq C\sum_{i=1}^{n} |a_i|, \text{ for all } n \in \mathbb{N}.
\]
In particular, taking $n=2$, $a_1=1$, $a_2=-1$ in the previous inequality, we obtain that $(\gamma_n)$ is 2-separated:
\[
\forall i \neq j, \ \|\gamma_i-\gamma_j\| \geq 2.
\]
For simplicity, we will use the same notation $(\gamma_n)$, $(x_n)$, and $(y_n)$ throughout the proof, even when we pass to subsequences.

\medskip

Up to passing to a subsequence, it is enough to consider the following four cases:
\begin{enumerate}
    \item[(1)] $(x_n)$ and $(y_n)$ are $\varepsilon$-separated for some $\varepsilon>0$;
    \item[(2)] $(x_n)$ is $\varepsilon$-separated for some $\varepsilon>0$, and $(y_n)$ converges to some $y \in M$;
    \item[(3)] $(x_n)$ converges to some $x \in M$ and $(y_n)$ converges to some $y \in M$ with $x \neq y$;
    \item[(4)] $(x_n)$ and $(y_n)$ converge to some element $\ell \in M$.
\end{enumerate}
\medskip

Case (1): There exists $\varepsilon>0$ such that $(x_n)$ and $(y_n)$ are $\varepsilon$-separated.
\smallskip

Up to passing to a subsequence, we may assume that for all $n \neq m$,\\ 
$x_n \notin B(x_m,\frac{\varepsilon}{2}) \cup B(y_m, \frac{\varepsilon}{2})$ and 
$y_n \notin B(x_m,\frac{\varepsilon}{2}) \cup B(y_m, \frac{\varepsilon}{2})$.

Let $B_n=\{x_n,y_n\}$. 
Clearly, the $B_n$'s are pairwise disjoint. 
By the Hahn--Banach theorem, 
\[
\exists g_n \in \mathrm{Lip}_0(M), \|g_n\|_L=\|\gamma_n\| \leq \frac{1}{C}, \ \langle g_n,\gamma_n\rangle=1.
\]
Now, for $n \in \N$, let $f_n$ be defined by 
\[
f_n=\begin{cases}
    g_n & \text{on } B_n,\\[0.3em]
    0 & \text{on } M \setminus (B(x_n,\frac{\varepsilon}{4}) \cup B(y_n,\frac{\varepsilon}{4})).
\end{cases}
\]

Let us check that $f_n$ is Lipschitz on $S_n:=B_n \cup \bigl(M \setminus (B(x_n,\frac{\varepsilon}{4}) \cup B(y_n,\frac{\varepsilon}{4}))\bigr)$.
Let $n \in \N$ and $x,y \in S_n$.

If $x$ and $y$ are both in $B_n$, then the Lipschitz inequality holds because $g_n$ is Lipschitz. 
The same is true if $x$ and $y$ are both in $M \setminus (B(x_n,\frac{\varepsilon}{4}) \cup B(y_n,\frac{\varepsilon}{4}))$,
since $f_n(x)$ and $f_n(y)$ both vanish. 
Assume now that $x \in B_n$ and $y \in M \setminus (B(x_n,\frac{\varepsilon}{4}) \cup B(y_n,\frac{\varepsilon}{4}))$. 
Hence,
\[
|f_n(x)-f_n(y)| = |g_n(x)-g_n(0)| \leq \|g_n\|_L\, d(x,0).
\]
However, $\|g_n\|_L=\leq \frac 1 C $ and $\frac{\ep}{4} \leq d(x,y)$, so 
\[
\|g_n\|_L \leq \frac{4}{C\ep}d(x,y).
\]
Consequently, 
\[
|f_n(x)-f_n(y)| \leq \frac{4}{C\ep} d(x,0)d(x,y) \leq \frac{4}{C\ep}\mathrm{rad}(M)\, d(x,y).
\]
Hence $f_n$ is Lipschitz and $\|f_n\|_L \leq \frac{4}{C\varepsilon}\mathrm{rad}(M)$. 
We extend each function $f_n$ to a function $\widetilde{f_n}$ in $\mathrm{Lip}_0(M)$ using the McShane--Whitney extension theorem (Lemma \ref{lemma:McShane}). 
As a consequence, $(\widetilde{f_n})$ is a bounded sequence in $\mathrm{Lip}_0(M)$, with disjoint support, because the $B_n$ are pairwise disjoint. Furthermore, $(\gamma_n)$ is bounded because it is equivalent to the canonical basis of $\ell^1$, and it is biorthogonal to $(\widetilde{f_n})$. Indeed, for $n \neq m \in \mathbb{N}$, since $x_m,y_m \notin B(x_n,\frac{\varepsilon}{4}) \cup B(y_n,\frac{\varepsilon}{4})$, 
$$\langle \widetilde{f_n},\gamma_m \rangle=a_m\widetilde{f_n}(x_m)+b_n\widetilde{f_n}(y_m)=0+0=0.$$ Moreover, $\langle \widetilde{f_n},\gamma_n\rangle=\langle g_n,\gamma_n\rangle=1$.
Hence, thanks to Lemma~\ref{lemma:Projection}, we obtain a projection from $\mathcal{F}(M)$ onto $\overline{\mathrm{span}}\{\gamma_n: n \in \mathbb{N}\}$.
\medskip

Case (2): There exists $\varepsilon>0$ such that $(x_n)$ is $\varepsilon$-separated, and $(y_n)$ converges to some $y \in M$.

First, since $(\gamma_n)$ is 2-separated, Lemma \ref{lemma:sev} with $\delta=\frac{1}{2}$ ensures that, up to a further extraction:  
\[
\forall i \in \mathbb{N}, \ d(\gamma_i,\mathrm{span}\{\delta(y)\})> \frac 1 2.
\]

Up to passing to a subsequence, we may assume that
\[
\forall n \in \mathbb{N}, \quad \ y_n \in B(y,\varepsilon) \quad\text{and}\quad x_n \notin B\left(y,\frac{\ep}{2}\right).
\]
Up to a further extraction, we may assume that:
\[\ d(y,y_{n+1}) \leq \tfrac{1}{4}d(y,y_n), \text{ for all } n \in \mathbb{N}.
\]
Let us consider the sets 
\[
A_n=\left\{z \in M ~  | ~  d(z,y_n) \leq \tfrac{1}{4}d(y,y_n)\right\},
\]
and let us check that the $A_n$'s are disjoint.
Let $z \in A_n \cap A_m$ with $n<m$. Then
\[
d(y,y_n) \leq d(y_n,z)+d(z,y) \leq \tfrac{1}{4}d(y,y_n)+d(z,y_m)+d(y,y_m).
\]
Thus
\[\tfrac{3}{4}d(y,y_n) \leq \tfrac{5}{4}d(y,y_m).
\]
Hence
\[
4^{m-n}d(y,y_m) \leq d(y,y_n) \leq \tfrac{5}{3}d(y,y_m),
\]
which is absurd. Thus $A_n \cap A_m =\emptyset$.\\
By the Hahn--Banach theorem,
\[
\forall n \in \mathbb{N}, \ \exists g_n \in \mathrm{Lip}_0(M), \ \|g_n\|_L=2, \ 
\begin{cases}
g_n(y)=0,\\
g_n(\gamma_n)=1.
\end{cases}
\]

We now consider, for all $n \in \mathbb{N}$,
\[f_n=\begin{cases}
    g_n & \text{on } \{x_n,y_n,y\},\\[0.3em]
    0 & \text{on } M \setminus (B(x_n,\tfrac{\varepsilon}{4}) \cup A_n).
\end{cases}
\]
The $f_n$'s have pairwise disjoint supports.
Let $n \in \mathbb{N}$. We now check that $f_n$ is Lipschitz. 
Let $x,z$ be in the set where $f_n$ is defined. For the same reasons as in case (1), the only nontrivial verification is when 
$x \in M \setminus (B(x_n,\tfrac{\varepsilon}{4}) \cup A_n)$ and $z \in \{x_n,y_n,y\}$. 
\begin{enumerate}
    \item[-] If $z=x_n$, we proceed exactly as in case (1) (since $z \notin B(x_n,\tfrac{\varepsilon}{4})$) and find that $\|f_n\|_L \leq \tfrac{8}{\varepsilon}\mathrm{rad}(M)$
    \item[-] If $z=y_n$, then 
    \[
    |f_n(x)-f_n(y_n)| = |g_n(y)-g_n(y_n)| \leq 2d(y,y_n) 
    \leq 2 d(x,y_n).
    \]
    \item[-] If $z=y$, then $|f_n(x)-f_n(y)|=0.$
\end{enumerate}
Thus, in all cases, $f_n$ is Lipschitz. 
We extend each $f_n$ to $\widetilde{f_n}$ thanks to the McShane--Whitney extension Lemma (Lemma \ref{lemma:McShane}). 
Then $(\widetilde{f_n})$ is bounded in $\mathrm{Lip}_0(M)$ and $(\gamma_n)$ is bounded in $\mathcal{F}(M)$. Let us check that $(\gamma_n)$ is biorthogonal to $(\widetilde{f_n})$. Let $n \neq m$ in $\mathbb{N}$, since $x_m, y_m \notin B(x_n,\frac{\varepsilon}{4}) \cup A_n$: 
$$\langle \widetilde{f_n},\gamma_m \rangle=a_m\widetilde{f_n}(x_m)+b_n\widetilde{f_n}(y_m)=0+0=0.$$ 
Moreover, $\langle \widetilde{f_n},\gamma_n\rangle=\langle g_n,\gamma_n\rangle=1$.
By Lemma~\ref{lemma:Projection}, we obtain the desired projection.
\medskip

Case (3): $(x_n)$ converges to
a certain $x \in M$ and $(y_n)$ converges to a certain $y \in M$ with $x \neq y$.

Since $x \neq y$, we have $\varepsilon:=d(x,y)>0$. 
Passing to a subsequence, we may assume that
\[
\forall n \in \mathbb{N}, \; x_n \in B\left(x,\frac {\varepsilon}{4}\right) \text{ and } y_n \in B\left(y,\frac{\varepsilon}{4}\right).
\]
Let us consider the sets as in the previous case: 
\[
A_n=\left\{z \in M ~|~ d(z,x_n) \leq \tfrac{1}{4}d(x,x_n)\right\}, \quad 
B_n=\left\{z \in M ~|~ d(z,y_n) \leq \tfrac{1}{4}d(y,y_n)\right\}.
\]
Up to an extraction, we may assume that: 
\[
\ A_n \subset B\left(x,\tfrac{\varepsilon}{4}\right) \text{ and } B_n \subset B\left(y,\tfrac{\varepsilon}{4}\right), \text{ for all } n \in \mathbb{N}.
\]
The $A_n$'s are pairwise disjoint, the $B_n$'s are pairwise disjoint, and since the terms of $(x_n)$ and $(y_n)$ are well separated, we also have 
\[
\forall n \neq m, \ A_n \cap B_m = \emptyset.
\]
Furthermore, since $(\gamma_i)$ is 2-separated, by Lemma \ref{lemma:sev} we may assume that:
\[
\ 
d(\gamma_i,\mathrm{span}\{\delta(x),\delta(y)\})>\frac 12, \text{ for all } i \in \mathbb{N}.
\]
By the Hahn--Banach theorem, 
\[
\forall n \in \mathbb{N}, \ \exists g_n \in \mathrm{Lip}_0(M), \ \|g_n\|_L=2, \ 
\begin{cases}
    g_n(x)=g_n(y)=0,\\
    g_n(\gamma_n)=1.
\end{cases}
\]
We now consider, for all $n \in \mathbb{N}$,
\[
f_n=\begin{cases}
    g_n & \text{on } \{x_n,y_n,x,y\},\\[0.3em]
    0 & \text{on } M \setminus (A_n \cup B_n).
\end{cases}
\]
The function $f_n$ is Lipschitz. Indeed, let $z,w$ be in the set where $f_n$ is defined.
Assume that $z \in \{x_n\} \cup \{y_n\} \cup \{x\} \cup \{y\}$ and $w \in M\setminus (A_n \cup B_n)$ (the other cases are clear).
\begin{enumerate}
    \item[-] If $z=x_n$, then 
    \[
    |f_n(x_n)-f_n(w)| = |g_n(x_n)-g_n(x)| \leq 2 d(x_n,x) \leq  8d(x_n,w).
    \]
    \item[-] If $z=y_n$, a similar estimate holds: 
    \[
    |f_n(y_n)-f_n(w)| \leq 8d(y_n,w).
    \]
    \item[-] If $z=x$, then 
    \[
    |f_n(x)-f_n(w)| = |g_n(x)-f_n(w)| = 0.
    \]
    \item[-] If $z=y$, then 
    \[
    |f_n(y)-f_n(w)| = 0.
    \]
\end{enumerate}
 We extend each $f_n$ to $\widetilde{f_n} \in \mathrm{Lip}_0(M)$ using the McShane--Whitney extension theorem (Lemma \ref{lemma:McShane}). By construction, the supports of the $\widetilde{f_n}$'s are pairwise disjoint. The sequences $(\widetilde{f_n})$ and $(\gamma_n)$ are bounded, and they are biorthogonal for the same reasons as in the case 2. The desired projection follows from Lemma~\ref{lemma:Projection}.
\medskip

Case (4): $(x_n)$ and $(y_n)$ converge to the same element $\ell \in M$.

Up to passing to a subsequence, we may assume that, for every $n \in \mathbb{N}$:
\[ \ d(x_{n+1},\ell) \leq \tfrac{1}{4}\min(d(x_n,\ell),d(y_n,\ell)) 
\text{ and } d(y_{n+1},\ell)\leq \tfrac{1}{4}\min(d(x_n,\ell),d(y_n,\ell)).
\]
As in the previous cases, let us consider the sets
\[
A_n=\left\{z \in M ~|~ d(z,x_n) \leq \tfrac{1}{4}d(x_n,\ell)\right\}, 
\quad 
B_n=\left\{z \in M ~|~ d(z,y_n) \leq \tfrac{1}{4}d(y_n,\ell)\right\}.
\]
We put $C_n=A_n \cup B_n$. 
We now show that the $C_n$ are pairwise disjoint. Since the $A_n$ and $B_n$ are already pairwise disjoint, it remains to prove that
\[
\ A_n \cap B_m = \emptyset, \text{ for all } n \neq m.
\]
Let $n,m \in \mathbb{N}$ with $n<m$, and let $z \in A_n \cap B_m$. 
If $d(x_n,\ell) \leq d(y_n,\ell)$, then
\[
d(x_n,\ell) \leq d(x_n,z)+d(z,\ell) \leq \tfrac{1}{4}d(x_n,\ell)+d(z,y_m)+d(y_m,\ell),
\]
so
\[
\tfrac{3}{4}d(x_n,\ell) \leq \tfrac{5}{4}d(y_m,\ell),
\]
hence
\[
4^{m-n}d(y_m,\ell) \leq d(x_n,\ell) \leq \tfrac{5}{3}d(y_m,\ell),
\]
which is absurd.

If $d(y_n,\ell) \leq d(x_n,\ell)$, a similar argument yields
\[
4^{m-n}d(x_m,\ell) \leq d(y_n,\ell) \leq d(x_n,\ell) \leq \tfrac{5}{3}d(x_m,\ell),
\]
which is also absurd. 
We have thus shown that $A_n$ and $B_m$ are disjoint for $n \neq m$. 

We now proceed similarly as in the other cases: since $(\gamma_n)$ is 2-separated, we may assume that:
\[
\forall n \in \mathbb{N}, \ 
d(\gamma_n,\mathrm{span}\{\delta(\ell)\})>\frac 12.
\]
By the Hahn--Banach theorem,
\[
\forall n \in \mathbb{N}, \ \exists g_n \in \mathrm{Lip}_0(M), \ \|g_n\|_L=2, \ 
\begin{cases}
    g_n(\ell)=0,\\
    g_n(\gamma_n)=1.
\end{cases}
\]
Let, for all $n \in \mathbb{N}$,
\[
f_n=\begin{cases}
    g_n & \text{on } \{x_n,y_n,\ell\},\\[0.3em]
    0 & \text{on } M \setminus C_n.
\end{cases}
\]
By construction, the $f_n$'s have disjoint supports. Let us check that each $f_n$ is Lipschitz. As before, it suffices to consider $x \in M \setminus C_n$ and $y \in \{x_n,y_n,\ell\}$.
If $y=x_n$, then
\[
|f_n(x)-f_n(x_n)| = |g_n(\ell)-g_n(x_n)| \leq 2d(x_n,\ell) 
\leq 8 d(x,x_n).
\]
Exactly the same way, if $y=y_n$, it follows that
\[
|f_n(x)-f_n(y_n)| \leq 8d(x,y_n).
\]
Hence $f_n$ is Lipschitz and biorthogonal to $(\gamma_n)$. 
We extend it to $\widetilde{f_n}$ using the McShane--Whitney extension theorem (Lemma \ref{lemma:McShane}), 
and we check that the assumptions of Lemma~\ref{lemma:Projection} are satisfied, which yields the desired projection.
\end{proof}

Let us now state the corollary about the adjoint $\omega C_f$: 

\begin{cor}{} 
Let $M,N$ be infinite metric spaces, $f : M \longrightarrow N$ and $\omega : M \longrightarrow \mathbb{C}$ such that $\fonction{\omega C_f}{\mathrm{Lip}_0(N)}{\mathrm{Lip}_0(M)}{g}{x \longmapsto \omega(x)g(f(x))}$ is a bounded operator.

The following properties are equivalent:
\begin{enumerate}
    \item[(i)] $\omega C_f$ is compact;
    \item[(ii)] $\omega C_f$ is weakly compact;
    \item[(iii)] $\omega C_f$ is super strictly singular;
    \item[(iv)] $\omega C_f$ is super strictly cosingular;
    \item[(v)] $\omega C_f$ is strictly singular;
    \item[(vi)] $\omega C_f$ is strictly cosingular;
    \item[(vii)] $\omega C_f$ does not fix any copy of $\ell^{\infty}$;
    \item[(viii)] $\omega C_f$ does not fix any complemented copy of $\ell^{\infty}$.
\end{enumerate}

\end{cor}

\begin{proof}{}
According to Schauder's theorem (\cite{Megg}, Theorem 3.4.15), an operator is compact if and only if its adjoint is compact. Hence, $\omega\widehat{f}$ is compact if and only if $\omega C_f$ is compact.

In addition, Gantmacher's theorem (\cite{Megg}, Theorem 3.5.13) states that an operator is weakly compact if and only if its adjoint is weakly compact. Hence, $\omega\widehat{f}$ is weakly compact if and only if $\omega C_f$ is weakly compact.

Now, thanks to Theorem~\ref{thm:main}, we get the equivalence: 
$$\omega C_f \text{ is compact } \iff \omega C_f \text{ is weakly compact }.$$

So now, we have the same diagram of implications as Figure~\ref{diagram} for $\omega C_f$. It remains to show that:
$$\omega C_f \text{ doesn't fix any complemented copy of } \ell^{\infty} \Rightarrow  \omega C_f \text{ is compact. }$$
Let us reason by contraposition. If $\omega C_f$ is not compact, then $\omega\widehat{f}$ is not compact and by Theorem \ref{thm:main}, $\omega \widehat{f}$ does fix a complemented copy of $\ell^1$. There exist a complemented subspace $E$ of $\mathcal{F}(M)$ with $E \simeq \ell^1$, a complemented subspace $F$ of $\mathcal{F}(N)$ with $F \simeq \ell^1$, such that $\restriction{\omega\widehat{f}}{E} : E \longrightarrow F$ is an isomorphism. There exist two continuous projections $P_E : \mathcal{F}(M) \longrightarrow E$ and $P_F : \mathcal{F}(N) \longrightarrow F$. Since $P_E$ is surjective, the adjoint $P_E^{*} : E^* \longrightarrow \mathrm{Lip}_0(M)$ is an isomorphic embedding. Thus, we can see $E^*$ as a closed subspace of $\mathrm{Lip}_0(M)$: there exists $X$ a subspace of $\mathrm{Lip}_0(M)$ such that $E^* \simeq X$. Hence, $\ell^{\infty} \simeq E^* \simeq X$. The same reasoning with $P_F$ yields to the existence of a subspace $Y$ of $\mathrm{Lip}_0(N)$ such that $\ell^{\infty} \simeq F^* \simeq Y$. Consequently, the operator $\restriction{\omega C_f}{Y} : Y \longrightarrow X$ is an isomorphism since it is the adjoint of the isomorphism $\restriction{\omega\widehat{f}}{E}$. We showed that $\omega C_f$ fix a copy of $\ell^{\infty}$. To show that this copy is complemented, let us consider the embedding $i : E \hookrightarrow \mathcal{F}(M)$, and define $\fonction{G}{\mathrm{Lip}_0(M)}{X}{h}{P_E^*(i^*(h))}$. The linear map $G$ is continuous, let us show that it is a projection. For $h \in \mathrm{Lip}_0(M)$, $x \in M$, we have:
$$\langle G \circ G(h),\delta(x)\rangle=\langle P_E^* \circ i^* \circ P_E^* \circ i^*(h),\delta(x)\rangle=\langle g,i \circ P_E \circ i \circ P_E(\delta(x))\rangle.$$
Since $i$ is the canonical injection, and $P_E$ is a projection, it follows that: 
$$\langle G \circ G(h),\delta(x)\rangle=\langle h, i \circ P_E(\delta(x))\rangle=\langle P_E^* \circ i^*(h),\delta(x)\rangle=\langle G(h),\delta(x)\rangle.$$

\end{proof}

\section{Generalization to operators preserving finitely supported elements}
 
We now investigate whether this result can be extended to the case of a bounded operator $T$ between Lipschitz-free spaces that preserves finitely supported elements.

\begin{thm}\label{thm:secondmain}
Let $T: \mathcal{F}(M) \longrightarrow \mathcal{F}(N)$ be a bounded operator such that:
\[
\exists k \in \mathbb{N}, \quad T(\mathcal{FS}_2(M)) \subset \mathcal{FS}_k(N).
\]
Then the following properties are equivalent:
\begin{enumerate}
    \item[(i)] $T$ is compact;
    \item[(ii)] $T$ is weakly compact;
    \item[(iii)] $T$ is super strictly singular;
    \item[(iv)] $T$ is super strictly cosingular;
    \item[(v)] $T$ is strictly singular;
    \item[(vi)] $T$ is strictly cosingular;
    \item[(vii)] $T$ does not fix any copy of $\ell^1$;
    \item[(viii)] $T$ does not fix any complemented copy of $\ell^1$. 
\end{enumerate}
\end{thm}

Although this theorem is strictly stronger than Theorem~\ref{thm:main}, so that it would in principle suffice to prove the former and omit a proof of the latter, most examples we have in mind fall within the scope of the first statement, which is therefore the most relevant for our purposes. For reasons of clarity and convenience, we thus presented a complete proof of the first theorem and now only sketch the proof of the second.
\medskip

As in the proof of Theorem~\ref{thm:main}, it suffices to prove $(viii) \Rightarrow (i)$, which we do by the same reasoning.
Proceeding by contraposition, assume that $T$ is not compact. Hence, $\exists x_n,y_n \in M$ such that $(\gamma_n):=(T(m_{x_n,y_n}))_{n \in \mathbb{N}}$ does not have convergent subsequence.
Since $\mathcal{FS}_k(N)$ has the Schur property, $(\gamma_n)$ does not admit any weakly Cauchy subsequence. By Rosenthal's $\ell^1$ theorem, we may assume that $(\gamma_n)$ is equivalent to the canonical basis of $\ell^1$. The desired implication will follow from the next proposition.

\begin{prop}{}
Let, for all $n \in \mathbb{N}$, $\gamma_n=\sum\limits_{i=1}^{k} a_i^n\delta(x_i^n) \in \mathcal{FS}_k(M)$. We assume that $(\gamma_n)$ is equivalent to the canonical basis of $\ell^1$. Then, there exists a subsequence $(\gamma_{n_k})_k$ of $(\gamma_n)$ which is complemented in $\mathcal{F}(M)$. 
\end{prop}

\begin{proof}{}
In the whole proof, we will always use the notation $(\gamma_n)$ and $(x_i^n)$ even for potential extractions.\\
As in Proposition~\ref{prop:cl1}, we may assume that $M$ is bounded without loss of generality.
Up to a re-indexation, and up to passing to a subsequence, we may assume that:
\begin{itemize}[leftmargin=*]
    \item there are $j$ sequences $(x_1^n),...,(x_j^n)$ ($j \leq k$) which are $\varepsilon$-separated for some $\varepsilon>0$, 
    \item and $k-j$ sequences which converge in $M$: $x_{j+1}^n \rightarrow x_{j+1},..., x_k^n \rightarrow x_k$. 
\end{itemize}
Note that one (and only one) of the integers $j$ or $k-j$ may be 0. In that case, some of the sets we are going to consider will be empty but this will cause no difficulty.\\
Up to an extraction, we can assume that: 
$$\forall i \in  \{1,...,j\}, \forall n  \neq m, x_i^n \notin \bigcup_{\underset{p \neq i}{p=1}}^{j} B\left(x_p^m,\frac{\varepsilon}{4}\right),$$
$$\forall i \in \{j+1,...,k\}, x_i^n \in B\left(x_i,\frac{\varepsilon}{4}\right).$$
Moreover, we may assume that $\displaystyle\bigcup_{\underset{p \neq i}{p=1}}^{j} B\left(x_p^m,\frac{\varepsilon}{4}\right)$ and $B\left(x_i,\frac{\varepsilon}{4}\right)$ for $i \in \{j+1,...,k \}$, as well as the balls $B\left(x_i,\frac{\varepsilon}{4}\right)$ and $B\left(x_{i'},\frac{\varepsilon}{4}\right)$ for $i,i' \in \{j+1,...,k \}$, $x_i\neq x_{i'}$, are pairwise disjoint.\\
Let $B_n=\{x_1^n,x_2^n,...,x_j^n\}$. We may assume that the $B_n$'s are pairwise disjoint. 
We also know that $(\gamma_n)$ is equivalent to the canonical basis of $\ell^1$, so up to multiplying by a positive constant we may assume that $(\gamma_n)$ is 2-separated: 
$$\forall i \neq j, \|\gamma_i-\gamma_j\| \geq 2.$$

Let $C_n=\{x_{j+1}^n,x_{j+2}^n,...,x_k^n\}$. 
Up to a further extraction, we may assume that: 
$$  d(x_i,x_i^{n+1}) \leq \frac{1}{4}\min_{p \in \{ j+1,...,k\}}d(x_p,x_p^n), \text{ for all } n \in \mathbb{N}, \text{ for all } i \in \{j+1,...,k\}.$$
For all $n \in \mathbb{N}$, for all $i \in \{ j+1,...,k\}$, let us consider the sets: 
$$A_i^n=\left\{z \in M ~  | ~ d(z,x_i^n) \leq \frac{1}{4}d(x_i,x_i^n)\right\}.$$
Up to an extraction, we may assume that: $\forall n \in \mathbb{N}, \forall i \in \{j+1,...,k\}, A_i^n \subset B\left(x_i,\frac{\varepsilon}{4}\right)$. It is readily seen that the $A_i^n$'s are all pairwise disjoints (see the proof of the cases (3) and (4) in Proposition~\ref{prop:cl1}).
Let us consider $F=\mathrm{span}\{\delta(x_i), i \in \{ j+1,...,k\}\}$. Applying Lemma \ref{lemma:sev} with $\delta=\frac{1}{2}$, one has, up to passing to a subsequence:
$$ d(\gamma_n,\mathrm{span}\{\delta(x_i)\})>\frac{1}{2} \text{ for all } n \in \mathbb{N}, \text{ for all } i \in \{ j+1,...,k\}.$$
By the Hahn-Banach theorem, 
$$\forall n \in \mathbb{N}, \exists h_n \in \text{Lip}_0(M), \|h_n\|_L=2, \begin{cases}
    \restriction{h_n}{F} \equiv 0\\
    \langle h_n,\gamma_n\rangle=1.
\end{cases}$$
Finally, let, for all $n \in \mathbb{N}$, 
$$f_n=\begin{cases}
    h_n \text{ on } B_n \cup C_n \cup \{x_{j+1},...,x_k\}\\
    0 \text{ on } M \setminus \left(\displaystyle\bigcup_{i=1}^{j} B \left(x_i^n,\frac{\varepsilon}{4}\right) \cup \bigcup_{i=j+1}^{k} A_i^n\right).
\end{cases}$$
We must check that for $n \in \mathbb{N}$, $f_n$ is Lipschitz on the set on which it is defined. Let $x,z$ in this reunion. Several cases must be considered:

\begin{enumerate}
    \item[-] If $x,z$ are both in $B_n \cup C_n \cup \{x_{j+1},...,x_k\}$, then the Lipschitz inequality holds since $h_n$ is Lipschitz.
    \item[-] If $x,z$ are both in $M \setminus \left(\displaystyle\bigcup_{i=1}^{j} B \left(x_i^n,\frac{\varepsilon}{4}\right) \cup \bigcup_{i=j+1}^{k} A_i^n\right)$, then $f_n(x)=f_n(z)=0$.
    \item[-] If $x \in B_n \cup C_n \cup \{x_{j+1},...,x_k\}$ and $z \in M \setminus \left(\displaystyle\bigcup_{i=1}^{j} B \left(x_i^n,\frac{\varepsilon}{4}\right) \cup \bigcup_{i=j+1}^{k} A_i^n\right)$, 
    \begin{enumerate}
        \item[·] If $x=x_i^{n}$ for $i \in \{ 1,...,j\}$, we pursue as in the case (1) of the proof of Proposition \ref{prop:complemented}:
        $$|f_n(x)-f_n(z)|=|h_n(x_i^n)-h_n(0)| \leq \|h_n\|_Ld(x_i^n,0) \leq \frac{8}{\varepsilon}\mathrm{rad}(M)d(x,z).$$
        \item[·] If $x=x_i^n$, $i \in \{j+1,...,k \}$, we pursue as in the case (4) of the proof of Proposition \ref{prop:complemented}:
        $$|f_n(x)-f_n(z)|=|h_n(x_i^n)-h_n(x_i)| \leq 2d(x_i^n,x_i) \leq 8d(z,x_i^n).$$
        \item[·] If $x=x_i$ for $i \in \{ j+1,...,k\}$, we simply have:
        $$|f_n(x)-f_n(z)|=|h_n(x_i)-h_n(x_i)|=0.$$
    \end{enumerate}
\end{enumerate}
Hence, $f_n$ is Lipschitz, and using the McShane-Whitney extension theorem (Lemma \ref{lemma:McShane}), we can extend it to obtain $\widetilde{f_n} \in \text{Lip}_0(M)$. The sequence $(\widetilde{f_n})_n$ is bounded in $\mathrm{Lip}_0(M)$ and have disjoint supports. The sequence $(\gamma_n)$ is bounded in $\mathcal{F}(M)$ (because equivalent to the canonical basis of $\ell^1$) and biorthogonal to $(\widetilde{f_n})$
We use the lemma \ref{lemma:Projection} and we obtain the desired projection.
\end{proof}

There are many operators satisfying the hypotheses of Theorem~\ref{thm:secondmain}. As already observed, it covers the classical Lipschitz operators $\widehat{f}$ as well as the weighted Lipschitz operators $\omega \widehat{f}$, which were already encompassed by Theorem~\ref{thm:main}. Note that the multiplication operators
\[
\delta(x) \in \mathcal{F}(M) \mapsto \omega(x)\,\delta(x),
\]
for $\omega \colon M \to \mathbb{R}$, are special cases of weighted Lipschitz operators $\omega \widehat{f}$ with $f = \mathrm{Id}_M$. In particular, the operators appearing in Kalton's decomposition (see \cite{Kalton}, Proposition~4.3) satisfy the hypotheses of Theorem~\ref{thm:secondmain}, as does the operator constructed in \cite{AA} in the proof of their Theorem~A.

\begin{remark}
Let us briefly point out that Theorem~\ref{thm:secondmain} is genuinely stronger than Theorem~\ref{thm:main}: there exist operators satisfying the hypotheses of Theorem~\ref{thm:secondmain} but not those of Theorem~\ref{thm:main}. For instance, it is natural to consider operators $T:\mathcal{F}(M)\to\mathcal{F}(N)$ defined by
\[
T(\delta(x)) = \sum_{i=1}^r \alpha_i(x)\,\delta(z_i(x)), 
\]
where the maps $\alpha_i : M \to \mathbb{R}$ and $z_i : M \to N$ are Lipschitz. Such an operator sends each Dirac mass at $x$ to a finite ``splitting'' of this mass among the points $z_1(x),\dots,z_r(x)$ in a Lipschitz-controlled way, and need not be of the form $\omega\,\widehat{f}$. From the viewpoint of optimal transport, which is closely related to Lipschitz-free spaces (the $1$-Wasserstein space over $M$ embeds isometrically into $\mathcal{F}(M)$; see e.g. \cite[Chapter~3.3]{Weaver2}), this reflects the fact that an optimal transport plan typically splits the mass at a given source point between several target points. Thus, the additional generality of Theorem~\ref{thm:secondmain} may be relevant from this perspective, even though we shall not pursue this direction here.

Another interesting example could be the following one. Let $(M,0_M)$ a pointed metric space and $f : M \longrightarrow M$ Lipschitz. Let us set the pointed metric space $(M,f(0_M))$. The fundamental factorization theorem ensures the existence of a bounded operator between the associate Lipschitz-free spaces $\widehat{f} : \mathcal{F}_{0_M}(M) \longrightarrow \mathcal{F}_{f(0_m)}(M)$ such that the following diagram commutes:
$$\begin{tikzcd}
(M,0_M) \arrow[d,"\delta_M"'] \arrow[r,"f"] & (M,f(0_M)) \arrow[d,"\delta_M"]\\
\mathcal{F}_{0_M}(M) \arrow[r,"\hat{f}"'] & \mathcal{F}_{f(0_M)}(M)
\end{tikzcd}$$
It is known that the construction of the Lipschitz-free spaces does not depend on the fixed origin, in the sense that there exists an isometric isomorphism $I$ between $\mathcal{F}_{f(0_M)}(M)$ and $\mathcal{F}_{0_M}(M)$ given by $I : \delta(x) \longmapsto \delta(x)-\delta(f(0_M))$. Let us now define $T_{f,0_M}:=I \circ \widehat{f} : \mathcal{F}_{0_M}(M) \longrightarrow \mathcal{F}_{0_M}(M)$, and show that $T_{f,0_M}$ verifies the hypothesis of Theorem \ref{thm:secondmain}. Let $a,b \in \mathbb{C}$ and $x,y \in M$. We have:
\begin{align*}
    I \circ \widehat{f}(a\delta(x)+b\delta(y))&=aI \circ \widehat{f}(\delta(x))+bI \circ \widehat{f}(\delta(y))\\
    &=aI(\delta(f(x)))+bI(\delta(f(y)))\\
    &=a(\delta(f(x))-\delta(f(0_M))+b(\delta(f(y))-\delta(f(0_M)) \in \mathcal{FS}_3(M)
\end{align*}
Such an operator $T_{f,0_M}$ is interesting for its linear dynamical properties.
\end{remark}

\begin{remark} \label{remarkeq}
There are probably many other operator properties that could be incorporated into Theorem~\ref{thm:secondmain}. In fact, by the structure of our proof, any operator property which is weaker than compactness but stronger than not fixing a (complemented) copy of $\ell_1$ will, under the assumptions of Theorem~\ref{thm:secondmain}, turn out to be equivalent to compactness. For the sake of completeness, let us mention two additional properties.
\begin{itemize}[leftmargin=*]
    \item Weakly precompact operators.
    A bounded operator $T\colon X \to Y$ is said to be \emph{weakly precompact} if every bounded sequence $(x_n)_n \subset X$ has a subsequence $(x_{n_k})_k$ such that $(T(x_{n_k}))_k$ is weakly Cauchy; equivalently, $T(B_X)$ is weakly precompact. By Rosenthal's $\ell_1$-theorem, a bounded set is weakly precompact if and only if it contains no sequence equivalent to the unit vector basis of $\ell_1$. It is then clear that
    \[
        T \text{ compact } \implies T \text{ weakly precompact } \implies T \text{ does not fix a copy of } \ell_1.
    \]
    Thus, weakly precompact operators can also be treated within the framework of Theorem~\ref{thm:secondmain}.
    \medskip

    \item Operators sending the unit ball to $V^*$-sets.
    Recall that a series $\sum\limits_{n=1}^\infty x_n$ in $X$ is said to be \emph{weakly unconditionally Cauchy} (abbreviated \emph{WUC}) if
    \[
        \sum_{n=1}^\infty \bigl|\langle x^*,x_n\rangle\bigr| < \infty 
        \quad\text{for every } x^* \in X^*.
    \]
    A subset $W \subset X$ is called a \emph{V*-set} if for every WUC series $\sum x_n^*$ in $X^*$ one has
    \[
        \lim_{n\to\infty} \sup_{x \in W} \bigl|\langle x, x_n^* \rangle\bigr| = 0.
    \]
    Every relatively weakly compact set is a V*-set, while the converse fails in general. It is well known (see, e.g., \cite[Proposition~1.1]{Bombal}) that $W \subset X$ is a V*-set if and only if it does not contain a sequence $(x_n)$ equivalent to the unit vector basis of $\ell_1$ whose closed linear span is complemented in $X$. Thus, if $T\colon \mathcal{F}(M) \to \mathcal{F}(N)$ satisfies the assumptions of Theorem~\ref{thm:secondmain}, then
    \[
        T\colon \mathcal{F}(M) \to \mathcal{F}(N) \text{ is compact}
        \quad\Longleftrightarrow\quad
        T(B_{\mathcal{F}(M)}) \text{ is a V*-set}.
    \]
\end{itemize}
\end{remark}

\begin{remark}
A Banach space $X$ is said to have \emph{Pe\l{}czy\'nski's property~(V*)} if every V*-set in $X$ is relatively weakly compact. Similarly, a bounded operator $T\colon X \to Y$ is called a \emph{V*-operator} if it maps V*-sets to relatively weakly compact sets. Note that $X$ has property~(V*) if and only if the identity map $\mathrm{Id}_X \colon X \to X$ is a V*-operator. Since a compact operator sends bounded sets to relatively compact sets, and since V*-sets are bounded, it follows that every compact operator is a V*-operator.

However, in contrast with the situation described in Remark~\ref{remarkeq} above, it is not true that, for every $T\colon \mathcal{F}(M) \to \mathcal{F}(N)$ satisfying the assumptions of Theorem~\ref{thm:secondmain}, if $T$ is a V*-operator then $T$ is compact. Indeed, a simple counterexample is given by the identity map $\mathrm{Id}_{\mathbb{R}}\colon \mathbb{R} \to \mathbb{R}$. In this case, $\mathcal{F}(\mathbb{R})$ is linearly isometric to $L^1(\mathbb{R})$, a space which is well known to have property~(V*). Therefore
\[
\widehat{\mathrm{Id}_{\mathbb{R}}}
= \mathrm{Id}_{\mathcal{F}(\mathbb{R})}\colon \mathcal{F}(\mathbb{R}) \to \mathcal{F}(\mathbb{R})
\]
is a V*-operator, but it is obviously not compact. It is noteworthy that this operator is not Dunford--Pettis either, that is, it does not send relatively weakly compact sets to relatively compact sets. The latter class of Lipschitz operators has been recently characterized in \cite{p1u}.
\end{remark}

In view of the above remark, as well as the recent works \cite{ACP21,p1u} which provide metric characterizations of compact and Dunford--Pettis Lipschitz operators, the following question naturally arises.

\begin{question}
Is there a metric characterization, that is, one in terms of metric properties of $f: M \to N$, of those Lipschitz operators $\widehat{f}: \mathcal{F}(M) \to \mathcal{F}(N)$ which are V*-operators?
\end{question}

\subsection*{Acknowledgments} I would like to express my sincere thanks to Colin Petitjean for many helpful discussions and for his careful reading of this work.

\end{document}